\documentclass[11pt]{birkjour}
\newtheorem{thm}{Theorem}[section]

\newtheorem{lem}[thm]{Lemma}
\newtheorem{prop}[thm]{Proposition}
\theoremstyle{definition}

\theoremstyle{remark}

\numberwithin{equation}{section}

\begin{document}
	
	%
	%
	%
	%
	%
	%
	%
	%
	%

	\author[Jafar Ahmed,  Ajebbar Omar and Elqorachi Elhoucien]{Jafar Ahmed$^{1}$, Ajebbar Omar$^{2}$ and Elqorachi Elhoucien$^{1}$}
	
	\address{%
		$^1$
	 Ibn Zohr University, Faculty of sciences,
		Department of mathematics,
		Agadir,
		Morocco}
	\address{%
	$^2$
	Sultan Moulay Slimane University, Multidisciplinary Faculty,
	Department of mathematics and computer science,
	Beni Mellal,
	Morocco}
\email{hamadabenali2@gmail.com, omar-ajb@hotmail.com, elqorachi@hotmail.com }

	\title[  A Kannappan-sine addition law on semigroups]{  A Kannappan-sine addition law on  semigroups}
	\begin{abstract} Let $S$ be a semigroup and $z_{0}$  a fixed element in $S.$
		We determine the complex-valued solutions of  the following  Kannappan-sine addition law	$f(xyz_{0})=f(x)g(y)+f(y)g(x),x,y\in S.$
	\end{abstract}
	\maketitle\emph{}
		 \section{Introduction}
		 Let $S$ be a semigroup. Let $\mathbb{C}$ be a field of complex numbers. The sine addition law   for unknown functions $f,g:S \longrightarrow \mathbb{C}$ is the functional equation
		 \begin{equation}
		 \label{1}
		 f(xy)=f(x)g(y)+f(y)g(x), \hspace{0.4cm}x,y \in S,
		 \end{equation}
		  The functional equation (\ref{1})   has been solved  on   groups and semigroups generated by their squares (see for example \cite{ajjj}, \cite{eb} and \cite{ebb}).
		    The solutions of the functional equation (\ref{1})  on a general semigroup  were  given recently by Ebanks \cite[Theorem 3.5]{c},  \cite[ Theorem 2.1  ]{b} and \cite[ Theorem 3.1]{d}.  Moreover, the continuous solutions are also known on topological groups and semigroups.
		     For additional discussions  of this functional equation and their history, see \cite[Chapter 13]{a}, \cite[Chapter 4]{g}
		     and their references.
		     \\In this paper we consider the following  Kannappan-sine addition law (\ref{1}), namely
		     \begin{equation}
		     \label{3}
		     f(xyz_0)=f(x)g(y)+f(y)g(x), \; x,y \in S,
		     \end{equation}
		      where $S$ is a semigroup and $z_0$ is a fixed element in   $S$. \\ On a monoid with identity element $e$ the functional equation (\ref{3}) becomes $g(e)f(xy)+f(e)g(xy)=f(x)g(y)+f(y)g(x)$. So, if $f(e)=0$ or $g(e)=0$, the last functional equation is the sine addition law, which was solved recently by Ebanks on semigroups \cite{d,c}. When $f(e)=\beta\neq 0$ and $g(e)=\alpha\neq 0$ we can  easily verify that $f,g$ satisfy the following cosine addition law\\ $2\bigl[\dfrac{f}{\beta}+\dfrac{g}{\alpha}](xy)=[\dfrac{f}{\beta}+\dfrac{g}{\alpha}](x)[\dfrac{f}{\beta}+\dfrac{g}{\alpha}](y)-
		      [\dfrac{-f}{\beta}+\dfrac{g}{\alpha}](x)[\dfrac{-f}{\beta}+\dfrac{g}{\alpha}](y)$, and from \cite{b} we get the solutions of the  functional equation (\ref{3}) on monoids.  \\It should be noted that  the traitement above on monoids are used in \cite[Corollary 2.4]{l} to solve equation (\ref{3}) on groups.\\We notice also that, when $z_{0}$ is in the centre of $S$, then equation (\ref{3}) can be written as follows $f(x\ast y)=f(x)g(y)+f(y)g(x)$, where $x\ast y=xyz_0$ is an associative law on $S$, and then the
		      functional equations (\ref{1}) and (\ref{3}) coincide,
		      and as mentioned above their solutions are known, so the contributions
		      of the present paper to the theory of trigonometric functional equations
		      in the non-abelian case.\\
		 	 The special case of (\ref{3}) in which $g=\dfrac{1}{2}f$, is the functional equation
		 	\begin{equation}
		 	\label{7}
		 	f(xyz_0)=f(x)f(y), \hspace{0.4cm} x,y \in S,
		 	\end{equation}
		 	 which has been solved on semigroups by Stetk\ae r \cite[Proposition 16]{i}. He proved that any non-zero solution
		 	 $f: S \longrightarrow \mathbb{C}$ of (\ref{7}) has the form
		 	$f=\chi(z_0)\chi,$
		 	where $\chi$ is  exponential function on $S$ such that $ \chi(z_0) \neq 0$. We shall use his result in the present paper.\\ The result by Kannappan \cite{Kannp} and its generalization (see for example \cite{sahoo}, \cite{i}), has been an inspiration
		 	by its treatment of similar functional equations on groups.\\ Our main goal is to extend
		 	 works in the literature in which $S$ is a monoid or a group, by showing
		 	 that existence of an identity element or a group inversion are not needed,
		 	 because we present proofs that work for general semigroups. In this context
		 	 we obtain in Theorem \ref{t1}  explicit formulas for the solutions of (\ref{3}).
		 	 \\Our main contributions to the knowledge about   (\ref{3})
		    		are the following:\\
		    		1.  We extend the setting from groups to semigroups.\\
		    	2.  We relate the solutions of (\ref{3}) to those of (\ref{1}) and (\ref{7}).
		    			    		\\ We also show that the functional equation
		    		\begin{equation}
		    		\label{5}
		    		f(xyz_0)=f(x)g(y)-f(y)g(x), \hspace{0.4cm} x,y \in S
		    		\end{equation}
		    		has  only trivial solutions on a semigroup $S$, with $z_0$ is a fixed element in $S$.
		 	\section{Set up, notations and terminology}
		 	Throughout this paper we enforce the set up below.\\
		 		$S$ is a semigroup (a set with an associative composition rule) and $z_0$ is a fixed element in $S$.
		 		
		 	If $X$ is a
		 	topological space we denote by $C(X)$ the algebra of continuous functions
		 	from $X$ to the field of complex numbers  $\mathbb{C}$. Let $\mathbb{C}^{*}=\setminus\{0\}$.
		 	
		 	A map $A:S\longrightarrow \mathbb{C}$ is said to be additive if
		 	\begin{center}
		 	$
		 	A(xy)=A(x)+A(y)$   for   all $  x,y \in S$,\end{center}
		  and a function $\chi :S\longrightarrow \mathbb{C}$ is multiplicative if
		  \begin{center}
		   $
		 	\chi(xy)=\chi(x)\chi(y)$ for  all $ x,y \in S.$
		 \end{center}
		 	If $\chi  $ is multiplicative and  $\chi \neq 0$
		 	then  we call  $ \chi$ an \textit{exponential}. For a multiplicative function  $\chi :S\longrightarrow \mathbb{C}$,  we define the \textit{null space} $I_{\chi}$ by
		 	\begin{center}
		 	$I_{\chi} := \{x\in S\hspace{0.1cm} | \hspace{0.1cm} \chi(x) = 0\}$.
		 \end{center}
		 	Then $I_{\chi}$ is  either empty or a proper subset of $S$ and $I_{\chi}$ is a
		 	two sided ideal in S if not empty and $S\setminus I_{\chi}$ is a subsemigroup of $S$.
		 	
		 	 A function $f : S\longrightarrow \mathbb{C}$ is said to be central if $f(xy) = f(yx)$ for all $x, y \in S$.
		 	 A function $f$ is said to be abelian if
		 	 		 	 $
		 	 		 	  f(x_1 x_2,...,   x_n) = f(x_{\sigma(1)}x_{\sigma(2)},...,
		 	  x_{\sigma(n)})
		 	  $
		 	   for all
		 	 $x_1, x_2, . . . , x_n \in S$, all permutations $\sigma$ of $n$ elements and all $n \in \mathbb{N}$.\\ Note that additive and multiplicative functions are abelian.
		 	
		 	  For any  subset $T \subseteq S$ let  $T^2 :=\{xy  \hspace{0.1cm}| \hspace{0.1cm}x,y   \in T\}$ and for any fixed element $z_0$ in S we let
		 	 $T^2z_0:=\{xyz_0 \hspace{0.1cm}|\hspace{0.1cm} x,y\in T  \}$.
		 	
		 	To characterize and describe some solutions of our functional equations, we use the partition of the null space $I_{\chi}$ into the disjoint union $I_{\chi} =  P_{\chi}\cup (I_{\chi} \setminus P_{\chi} ) $, with
		 	 \begin{equation*}
		   P_{\chi}:=\{p\in I_{\chi}\setminus I_{\chi}^2 \hspace{0.1cm} | \hspace{0.1cm}up,pv,upv \in  I_{\chi}\setminus I_{\chi}^2 \hspace{0.1cm}\text{for}\hspace{0.1cm} \text{all} \hspace{0.1cm} u,v \in S\setminus I_{\chi}\}.
		 	  \end{equation*}
		 	  For more discussions   about  the partition of the null space $I_{\chi}$ we  refer to  \cite{ c}, \cite{ b} and \cite{d}. The following   lemma   will be used later.
		 	  \begin{lem}
		 	  	\label{eb}  $p\in P_{\chi} \implies up,pv, upv \in  P_{\chi}$ for all  $u,v \in
		 	    S\setminus I_{\chi}\ $.
		 	  	\end{lem}
		 	  	\begin{proof}
		 	  	Follows directly from the definition of $  P_{\chi}$.
		 	  	\end{proof}
		 	    \section{Preparatory works}
		 	   In this section  we prove some useful lemmas that will be used in the proof of  our main results.
		 	  	 \begin{lem}
		 	  	 	\label{ui}
		 	  	   Let $S$ be  a semigroup, $n\in \textit{N},$
		 	  	   and $ \chi_1,\chi_2,...\chi_n:S\longrightarrow \mathbb{C}$    be   exponential functions. Then  \\
		 	  	    (a) $\{\chi_1,\chi_2,...,\chi_n\}$ is linearly independent.  \\
		 	  	   (b) If  $A: S\setminus I_{\chi}\longrightarrow \mathbb{C}$ is  a non-zero additive function and $\chi$ is exponential on $S$, then    the set   $\{\chi A, \chi\}$ is linearly independent on   $S\setminus I_{\chi}$.
		 	  	 \end{lem}
		 	  	 \begin{proof}
		 	  	 	(a): See \cite [Theorem 3.18]{g}. (b): See \cite[Lemma 4.4]{ajjj}.
		 	  	 	\end{proof}
	 	  	 	To obtain some  of the solutions of the functional equation (\ref{3})
		 	 we   need   the solutions of the following functional equation
		 	  \begin{equation}
		 	  \label{78}
		 	  f(xyz_0)=f(x)+f(y), \hspace{0.3cm} x,y \in S.
		 	  \end{equation}
		 	    Lemma \ref{hh} describes the  solutions of the functional equation (\ref{78}) on semigroups.
		 	  \begin{lem}
		 	  	\label{hh}
		 	  	Let $S$ be a semigroup. The solutions    of the functional  equation  (\ref{78}) are the function
		 	  	of the  form	$f=A+ A(z_0)$, where $A :S \longmapsto\mathbb{C}$ is additive.
		 	  	\end{lem}
		 	  \begin{proof}
		 	  	Let $x,y,z \in S $ be arbitrary. Using (\ref{78}) we find that
		 	 \begin{equation*}
		  	 	f((xy)z_0^2z_0)= f(xy)+f(z_0^2).
		 	  \end{equation*}
		 	  	Similarly we find that
		 	  \begin{equation*}
		 	  		f((xyz_0)z_0z_0)= f(xyz_0)+f(z_0)= f(x)+f(y)+f(z_0).
		 	  		\end{equation*}
		 	  			Comparing the previous equations  we get that
		 	  	 \begin{equation*}
		 	  	f(xy)=f(x)+f(y)+[f(z_0)-f(z_0^2)].
		 	    \end{equation*}
		 	  	Therefore, we see that    $A:=f+[f(z_0)-f(z_0^2)]  $ is additive.
		 	  		Now, using  (\ref{78})  and the fact that $A$ is additive  we obtain that
                 \begin{equation*}\begin{split}
		 	  		 0&=f(xyz_0)-f(x)-f(y) \\ &=A(xyz_0)-[f(z_0)-f(z_0^2)]-(A(x)  -[f(z_0)-f(z_0^2)]) \\
		 	  		&\quad-(A(y)  -[f(z_0)-f(z_0^2)])\\
		 	  		&=A(z_0)+ [f(z_0)-f(z_0^2)].
		 	  	\end{split}\end{equation*}
		 	  	Thus $ A(z_0)= -[f(z_0)-f(z_0^2)] $, and
		 	  	hence
		 	  	$f= A +A(z_0)$.
		 	  	The converse statement is trivial.
		 	  	This completes the  proof of Lemma \ref{hh}.
		 	  	\end{proof}
		 	  		\section{Solutions of the Kannappan-sine addition law on  semigroups}
		In this section we solve the Kannappan-sine addition law (\ref{3}) on semigroups. The following lemmas will be used  in the proof of Theorem \ref{t1}.

		 \begin{lem}
		 	\label{l1}
		 	Let $S$ be a semigroup, and let $f, g :S\longmapsto\mathbb{C}$ be a solution  of the functional equation (\ref{3}). Then we have the following.\\
		 (i) If $f(z_0)=0 $ then
		    \begin{equation}
		 \label{14}
	g(z_0^2)f(xy)= g(z_0)[f(x)g(y)+f(y)g(x)]-f(z_0^2)g(xy),  \hspace{0.3cm} x,y \in S.
		 \end{equation}
		 	(ii) If $f(z_0)\neq0$  then there exists $\gamma \in \mathbb{C}$ such that
		 	\begin{equation}
		 	\label{63}
		 	g(xyz_0)=g(x)g(y)+\gamma^2 f(x)f(y), \hspace{0.3cm} x,y \in S.
		 	\end{equation}
		 \end{lem}
		 \begin{proof}
		 (i)	Let $x,y,z\in$  $S$ be arbitrary. Using (\ref{3}) and the fact that $ f(z_0)=0$ we have
		 	\begin{equation*}
		 	f((xyz_0)z_0 z_0) = f(xyz_0)g(z_0)+f(z_0)g(xyz_0)
		 	=g(z_0)[f(x)g(y)+f(y)g(x)]
		   	\end{equation*}
		 		and
		 			\begin{equation*}
		 		f((xyz_0)z_0 z_0)	= f((xy)z_0^2 z_0)
		 			=  f(xy)g(z_0^2)+f(z_0^2)g(xy).
		 		\end{equation*}
		 		  Then we deduce that
		 		 	\begin{equation*}
		 		 	  f(xy)g(z_0^2)+f(z_0^2)g(xy)=g(z_0)[f(x)g(y)+f(y)g(x)],\hspace{0.3cm} x,y \in S.
		 		 	\end{equation*}
		 		 This proves (i).\\
		 		(ii)	Let $x,y\in S$ be arbitrary. By  using (\ref{3}) we get
		 			\begin{equation*}\begin{split}
		 		 	f((xyz_0)z_0z_0  )&=f(xyz_0)g(z_0)+f(z_0)g(xyz_0)\\
		 		 	&=g(z_0)[f(x)g(y)+f(y)g(x)]+f(z_0)g(xyz_0)\\
		 		 	&=g(z_0)f(x)g(y)+g(z_0)f(y)g(x)+f(z_0)g(xyz_0).
		 		 	\end{split}\end{equation*}
		 		 Similarly, we find
		 		 	\begin{equation*}\begin{split}
		 		 	f(x(yz_0^2)z_0 ) &= f(x)g(yz_0^2)+f(yz_0^2)g(x)\\
		 		 	&=f(x)g(yz_0^2)+g(x)[f(y)g(z_0)+f(z_0)g(y)]\\
		 		 	&=f(x)g(yz_0^2)+g(z_0)g(x)f(y)+f(z_0)g(x)g(y).
		 		 	\end{split}\end{equation*}
		 		 	Then  the associativity of the operation in $S$ implies
		 		 		\begin{equation}
		 		 		\label{54}
		 		 			 f(z_0)[g(xyz_0)-g(x)g(y)]=f(x)[g(yz_0^2)-g(y)g(z_0)].
		 		 		\end{equation}
		 		 		Since $f(z_0)\neq0$, dividing equation (\ref{54}) by $f(z_0)$ we get that
		 		 			\begin{equation}
		 		 			\label{55}
		 		 		 g(xyz_0)= g(x)g(y)+f(x)\phi(y),
		 		 			\end{equation}
		 		 			where
		 		 			 \begin{equation*}
		  				\phi(y):= \dfrac {g(yz_0^2)-g(y)g(z_0)} {f(z_0)},\hspace{0.2cm} y\in S.
		 		 			 \end{equation*}
		 		 				Now, substituting (\ref{55})  come back   into (\ref{54})  we get
		 		 				 \begin{equation*}
		 		 				f(z_0)[f(x)\phi(y)]=f(x)[f(y)\phi(z_0)].
		 		 			\end{equation*}
		 		 				By putting $x=z_0$ in the last equation, we deduce that
		 		 				 \begin{center}
		 		 				 $\phi(y):=c f(y), $  where  $c:=\dfrac{\phi(z_0)}{f(z_0) }$. \end{center}
		 		 				 Choosing $\gamma\in \mathbb{C}$ such that $\gamma^2=c$,  equation (\ref{55}) becomes
		 		 				 \begin{equation*}
		 		 				 g(xyz_0)= g(x)g(y)+ \gamma^2 f(x)f(y).
		 		 				 \end{equation*}
		 		 				 This proves part (ii) and completes the proof of Lemma \ref{l1}.
		 		 				 \end{proof}
	 		 				  \begin{lem}
		 		 				  		 	\label{l2}
		 		 				  		 	Let $S$ be a semigroup,  and let $f, g :S\longmapsto\mathbb{C}$ be a solution  of the functional equation (\ref{3}) such that $f\neq 0$  and  $f(z_0)=0$.   Then we have the following.\\
		 		 				  		 	(i)	If $f(z_0^2)\neq0$ then $g(z_0^2)=0$.\\
		 		 				  		 	 	(ii) If $\{f,g\}$ is linearly independent then   the following holds
		 		 				  		 		
		 		 				  		 		(a)  $g(z_0)\neq0.$
		 		 				  		 	
		 		 				  		 		(b)  If $f(z_0^2)=0
		 		 				  		 		$ then $ g(z_0^2)\neq0$.
		 		 				  		 		\end{lem}
		 		 				  		 \begin{proof}
		 		 				  		 	(i) Assume that  $f(z_0)=0$   and letting $x=y=$  $z_0$    in (\ref{3}) we get that
		 		 				  		 	\begin{equation*}
 		 		 				  		 		f(z_0^3)=  f(z_0)g(z_0)+f(z_0)g(z_0)=0.
		 		 				  		 	 	\end{equation*}
		 		 				  		 Similarly, letting $x=z_0^2$ and $y=z_0^2$ and then $x = z_0$ and $y = z_0^3
		 		 				  		 $ in (\ref{3}) yields
		 		 				  		 	 \begin{equation*}\begin{split}
		 		 				  		 		f(z_0^5)&=2f(z_0^2)g(z_0^2)\\
		 		 				  		 			& =f(z_0)g(z_0^3)+f(z_0^3)g(z_0) =0.
		 		 				  		 	\end{split}\end{equation*}
		 		 				  		 	Since $f(z_0^2)\neq0,$ we deduce that 	$g(z_0^2)=0$.\\
		 		 				  		 	(ii)-
		 		 				 (a) Assume  that $g(z_0)=0$. Making    the substitutions  $(x,yz_0 )$ and  $(x y,z_0 )$ in (\ref{3}) we get respectively  that
		 		 				  		 	\begin{equation*}
		 		 				  		 		f(xyz_0z_0)= f(x )g(yz_0)+f(yz_0)g(x) =f(xy)g(z_0)+f(z_0)g(xy)=0,
		 		 				  		 	\end{equation*}
		 		 				  		 which implies that
		 		 				  			\begin{equation*}
		 		 				  		 	f(x )g(yz_0)+f(yz_0)g(x)=0,  \; x,y\in S.
		 		 				  		 \end{equation*}
		 		 				  		 	Since $f$ and $g$ are linearly independent we get in particular that  $ f(yz_0)=0$ for all $y\in S$.  Now, by  (\ref{3}) we have
		 		 				  		 			\begin{equation}
		 		 				  		 			\label{03}
		 		 				  		 				f(x)g(y)=-f(y)g(x), \hspace{0.3cm} x,y\in S.
		 		 				  		 			\end{equation}
		 		 				  		 			Applying \cite[Exercice 1.1(b)]{g} we get  $f=0$ or $g=0$,  which is a contradiction because  $f$ and $g$ are linearly independent. So $g(z_0)\neq0$.\\
		 		 				  		 (ii)-(b) Assume, for the sake of contradiction, that  $g(z_0^2)=0$. Since $f(z_0)=0$, and $f(z_0^2)=0$ by hypothesis,  equation (\ref{14})  reduces to
		 		 				  		 \begin{center}
		 		 				  		$
		 		 				  		g(z_0) f(x)g(y)+g(z_0)f(y)g(x)=0,$ $x,y\in S,$
		 		 				  			\end{center}
		 		 				  		 which can be written as follows
		 		 				  		 \begin{center}
		 		 				  		$
		 		 				  		    f(x)g(y)=- f(y)g(x),$ $x,y\in S,$
		 		 				  		
		 		 				  		\end{center}
		 		 				  		since $g(z_0)\neq0$ by  (b)-(i). Then  $f=0$, which contradicts the fact that $f $ and $g$ are linearly independent. So $g(z_0^2) \neq0$. This completes the proof of Lemma  \ref{l2}.\end{proof} In the following Proposition we solve a special case of Kannappan-sine addition law on semigroups, namely
		 		 				  	\begin{equation}
		 		 				  	\label{bb}
		 		 				  	f(xyz_0)=\chi (z_0)f(x)\chi(y)+ \chi (z_0)f(y)\chi(x), \;x,y \in S,
		 		 				  	\end{equation}
		 		 				  	where $\chi $ is an exponential   function on $S$.
		 		 				   Proposition \ref{t2} is a generalization of Ebanks's work   on the special case
		 		 				  	of the sine addition law on monoids $(z_0=e)$ \cite[Theorem 3.1]{c}.
		 		 				  		 \begin{prop}
		 		 				  		 		\label{t2}
		 		 				  		 				Let $S$ be a semigroup  and   $\chi :S\longmapsto\mathbb{C}$ be an exponential  function such that $\chi(z_0)\neq0$.   If  $f :S\longmapsto\mathbb{C}$ is a solution of   functional equation (\ref{bb}), then
		 		 				  		 					\begin{equation}
		 		 				  		 				\label{15}
		 		 				  		 		 	f(x)=
		 		 				  		 				\begin{cases}
		 		 				  		 				\chi(x)(A(x)+A(z_0)) \hspace{0.5cm} for\hspace{0.3cm} x\in S \backslash I_{\chi} \\
		 		 				  		 				\rho(x)\hspace{2.8 cm} for\hspace{0.3cm} x\in    P_{\chi} \\
		 		 				  		 				0 \hspace{3.3 cm}for\hspace{0.4cm} x\in I_{\chi}\setminus P_{\chi},
		 		 				  		 				\end{cases}
		 		 				  		 				\end{equation}
		 		 				  		 				 where $A: S \setminus$
		 		 				  		 						$I_{\chi}\longrightarrow \mathbb{C}$
		 		 				  		 				is additive and    $\rho: P_{\chi}\longrightarrow \mathbb{C}$  is the restriction of $f$ to $P_{\chi} $.\\In addition,    $f$  is abelian and  satisfies the following conditions:\\
		 		 				  		 				(I) $f(xy) = f(yx)=0 $ for all $ x \in I_{\chi} \setminus P_{\chi} $ and $ y \in S \setminus I_{\chi}$.\\
		 		 				  		 				(II) If $x \in \{up, pv, upv\}
		 		 				  		 				$ with $p \in P_{\chi} $ and $u,v \in S \setminus I_{\chi}$, then  $x \in P_{\chi} $ and we have respectively
		 		 				  		 				$\rho(x)=	\rho(p) \chi(u)$, $\rho(x)=	\rho(p) \chi(v)$ or $\rho(x)=	\rho(p) \chi(uv)$.\\Conversely, the  function f  of the form (\ref{15}) define a solution of (\ref{bb}).\\Moreover, if $S$ is a topological semigroup and $f \in C(S)$, then   $\chi \in C(S), A \in C(S \setminus  I_{\chi})$
		 		 				  		 			   and
		 		 				  		 				$\rho \in C(P_{\chi})$.\end{prop} \begin{proof} If   $I_{\chi}=\emptyset$ then  $\chi$ vanishes nowhere, so the functional equation (\ref{bb}) reduces to    	\begin{equation}
		 		 				  		 				\dfrac{f(xyz_0)}{\chi(xyz_0)}=\dfrac{f(x)}{\chi(x)}+\dfrac{f(y)}{\chi(y)}, \hspace{0.4cm}      x,y \in S.
		 		 				  		 				\end{equation} According to Lemma \ref{hh} we get that $f/\chi =:A+A(z_0)$ where $A: S   \longrightarrow \mathbb{C}$ is additive. Then we deduce that $f=\chi(A+A(z_0))$. This is a special case of (\ref{15}) (here the middle and the bottom lines of (\ref{15}) are  disappearing). So from now on we may assume that  $I_{\chi}\neq\emptyset$.
		 		 				  		  		 For  all $ x,y \in S\setminus I_{\chi}$ we have $\chi(xyz_0)=\chi(x)\chi(y)\chi(z_0)\neq0$,  because $ \chi(z_0)\neq 0$. Therefore, dividing   equation (\ref{bb}) by $ \chi(xyz_0)$ we get that
		 		 				  		  		 \begin{equation*}
		 		 				  		  		 \dfrac{f(xyz_0)}{\chi(xyz_0)}=\dfrac{f(x)}{\chi(x)}+\dfrac{f(y)}{\chi(y)}, \hspace{0.4cm}      x,y \in S\setminus I_{\chi}.
		 		 				  		  		 \end{equation*}
		 		 				  		  		 By Lemma \ref{hh} we deduce that $f/\chi=:A+A(z_0)$, where $A: S\setminus I_{\chi}\longrightarrow \mathbb{C}$ is additive, then $f= \chi(A+A(z_0))$ and
		 		 				  		  		 we have the top line of (\ref{15}).

		 		 				  		  		 Let $x,y   \in  I_{\chi}$. Using (\ref{bb}) we find that
		 		 				  		  		 \begin{equation}
		 		 				  		  		 \label{ccc}
		 		 				  		  		 f((xy)(z_0)z_0)= f(xy)\chi(z_0)\chi(z_0)+f(z_0)\chi(xy)\chi(z_0)=f(xy)\chi(z_0)^2
		 		 				  		  		 \end{equation}
		 		 				  		  		 and
		 		 				  		  		 \begin{equation}
		 		 				  		  		 \label{nnm}
		 		 				  		  		 f(x (yz_0)z_0)= f(x)\chi(y)\chi(z_0)\chi(z_0)+f(yz_0)\chi(x)\chi(z_0)=0,
		 		 				  		  		 \end{equation}
		 		 				  		  		 since  $\chi(xy)=\chi(x)=\chi(y)=0$. By using (\ref{ccc}), (\ref{nnm}) and the associativity of the operation in  $S$ we have
		 		 				  		  		 $
		 		 				  		  		  f(xy)\chi(z_0)^2= 0,
		 		 				  		  		 $
		 		 				  		  		 which gives $f(xy)=0$ since $ \chi(z_0)\neq0$.  This proves the case $x\in I_{\chi}^2$  of the third line of formula (\ref{15}).
		 		 				  		  		
		 		 				  		  		 Next suppose $x\in I_{\chi} \setminus(I_{\chi}^2 \cup P_{\chi})$. Then there exist $u,v \in S \setminus I_{\chi}$ such that one of $ux, vx$ or $ uxv \in I_{\chi}^2$.    For $ux \in I_{\chi}^2 $, by using   (\ref{bb}) again we have
		 		 				  		  		 \begin{equation*}\begin{split}
		 		 				  		  		 	f(z_0(ux)z_0)&=f(z_0)\chi(u)\chi(x)\chi(z_0)+f(ux)\chi(z_0)\chi(z_0)= 0\\
		 		 				  		  		 	&=f((z_0u)(x)z_0)=f(z_0u)\chi(x)\chi(z_0)+f(x)\chi(z_0u)\chi(z_0)\\
		 		 				  		  		 	&=f(x)\chi(u)\chi(z_0)^2,
		 		 				  		  		\end{split}\end{equation*}
		 		 				  		  		 because $f(ux)=\chi(x)=0$. It follows  that $ f(x)\chi(u)\chi(z_0)^2 =0$ and hence $f(x)=0$ since $\chi(u)\neq0$ and $\chi(z_0) \neq0$. The case $xv \in I_{\chi}^2 $ can be treated similarly. For the case $uxv \in I_{\chi}^2 $, we have
		 		 				  		  		\begin{equation*}\begin{split}
		 		 				  		  		 	f( (uxv)(z_0)z_0)&=f(uxv)\chi(z_0) \chi(z_0)+f( z_0)\chi(uv)\chi(x)\chi(z_0)= 0\\
		 		 				  		  		 	&=f(( u)(xvz_0)z_0)=f( u)\chi(xvz_0)\chi(z_0)+f(xvz_0)\chi( u)\chi(z_0) \\
		 		 				  		  		 	&=f(u)\chi(vz_0)\chi(x)\chi(z_0)+\chi(uz_0)[f(x)\chi( v)\chi( z_0)\\&\quad+f(v)\chi( x)\chi( z_0)]\\
		 		 				  		  		 	&=f(x)\chi(uv)\chi(z_0)^2,
		 		 				  		  		 \end{split}\end{equation*}
		 		 				  		  		 because $f(uxv)=\chi(x)=0$. Consequently $f(x)=0$ since $\chi(uv)\neq0$ and $\chi(z_0) \neq0$.
		 		 				  		  		 Then the proof of the third line of (\ref{15}) is completed. For the bottom line of (\ref{15}) there is nothing left to prove.

		 		 				  		  		 For the condition (I), let   $x\in  I_{\chi}  \setminus P_{\chi}$ and $y\in  S \setminus I_{\chi}$. Using (\ref{bb}) we have
		 		 				  		  		 \begin{equation*}\begin{split}
		 		 				  		  		 	f( (xy)z_0z_0)&=f(xy)\chi(z_0)\chi(z_0)  +f( z_0)\chi(xy) \chi(z_0)=f(xy) \chi(z_0)^2  \\
		 		 				  		  		 	&=f( (x)(yz_0)z_0)=f( x)\chi(yz_0)\chi(z_0)+f( yz_0)\chi( x)\chi(z_0) \\
		 		 				  		  		 	&=0,
		 		 				  		  		 \end{split}\end{equation*}
		 		 				  		  		 which gives $f(xy) =0$, since $f(x)=0$ and $\chi(z_0) \neq0$. The verification of $f(yx) =0$ is similar.
		 		 				  		  		
		 		 				  		  		 For the condition (II), we begin  by   the case $x=up$ with $p\in  P_{\chi}$ and $u \in  S \setminus I_{\chi}$. Then by (\ref{bb}) we have
		 		 				  		  		 \begin{equation*}\begin{split}
		 		 				  		  		 	f( xz_0^2z_0)&=f(x )\chi(z_0^2)\chi(z_0)  +f( z_0^2)\chi(x ) \chi(z_0)=f(x ) \chi(z_0)^3\\
		 		 				  		  		 	&=f( (u)(pz_0^2)z_0)=f( u)\chi(pz_0^2)\chi(z_0)+f( pz_0^2)\chi( u)\chi(z_0)\\
		 		 				  		  		 	&=f( u)\chi(p) \chi(z_0)^3+ \chi( u)\chi(z_0)[f(p)\chi( z_0)\chi(z_0)\\
                                                    &\quad+f( z_0)\chi( p)\chi(z_0)]\\
		 		 				  		  		 	&=\rho(p) \chi( u) \chi( z_0)^3,
		 		 				  		  		  \end{split}\end{equation*}
		 		 				  		  		 because  $\chi(p)=0$. Therefore $f(x)=\rho(p) \chi( u)$ since $ \chi(z_0)\neq0$. In addition we have for any  $z, t \in  S \setminus I_{\chi}$, that $zx=zup=(zu)p, \; xt=upt= (up)t, \; zxt=zupt=(zu)pt   \in  I_{\chi} \setminus I_{\chi}^2 $, so $x\in  P_{\chi}$ and  $\rho(x)=\rho(p) \chi( u )$. The case $x=pv$ can be treated similarly.
		 		 				  		  		 For the case $x=upv$ with $p\in  P_{\chi}$ and $u,v \in  S \setminus I_{\chi}$, we have
		 		 				  		  		 \begin{equation*}\begin{split}
		 		 				  		  		 	f( x(z_0^2)z_0)&=f(x )\chi(z_0^2)\chi(z_0)  +f( z_0^2)\chi(x ) \chi(z_0)=f(x ) \chi(z_0)^3  \\
		 		 				  		  		 	&=f( (u)(pvz_0^2)z_0)=f( u)\chi(pvz_0^2)\chi(z_0)+f( pvz_0^2)\chi( u)\chi(z_0) \\
		 		 				  		  		 	&=f( u)\chi(pv) \chi(z_0)^3+ \chi( u)\chi(z_0)[f(p)\chi(v z_0)\chi(z_0)\\
                                                    &\quad+f(v z_0)\chi( p)\chi(z_0)] \\
		 		 				  		  		 	&=\rho(p) \chi( u)\chi( v) \chi( z_0)^3,
		 		 				  		  		\end{split}\end{equation*}
		 		 				  		  		 because $ \chi(p)=0$. Then  $f(x)=\rho(p) \chi( u)\chi( v)$ since  $\chi( z_0)\neq0$. In addition we have also for any $z, t \in  S \setminus I_{\chi}$, that $zx=zupv=(zu)pv,  xt=up(vt)= (upv)t,  zxt=zupvt=(zu)p(vt)   \in  I_{\chi}\setminus I_{\chi}^2 $, then $x\in  P_{\chi}$. Therefore  $\rho(x)=\rho(p) \chi( u  v)$   and this completes the proof of condition (II) of Proposition (\ref{t2}).
		 		 				  		  		
		 		 				  		  		 Conversely, suppose that $f:S\longrightarrow \mathbb{C}$ satisfies (\ref{15}),  $A:S\setminus I_{\chi}\longrightarrow \mathbb{C}$ is additive, $\rho :P_{\chi}\longrightarrow \mathbb{C}$ and conditions (I) and (II)  hold. We shall prove that $f $ satisfy (\ref{hh}). The case $x, y \in S\setminus I_{\chi}$  is straightforward
		 		 				  		  		 so we omit it. For  $x \in I_{\chi}\setminus  P_{\chi}$  and $y \in S\setminus I_{\chi}$ we have $yz_0 \in S\setminus I_{\chi}$. By condition (I) we have
		 		 				  		  		 \begin{equation*}
		 		 				  		  		 f(xyz_0)=f(x(yz_0))=0=f(x)\chi(y)\chi(z_0)+f(y)\chi(x)\chi(z_0),
		 		 				  		  		 \end{equation*}
		 		 				  		  		 since $f(x)=\chi(x)=0$. The case  $y \in I_{\chi\setminus} P_{\chi}$  and $x \in S\setminus I_{\chi}$ can be treated similarly.
		 		 				  		  		
		 		 				  		  		 For  $x \in P_{\chi}$  and $y \in S\setminus I_{\chi}$ we   have   $yz_0 \in S\setminus I_{\chi}$ and then by Lemma \ref{eb} we get $xyz_0 \in P_{\chi}$, then
		 		 				  		  		 \begin{equation*}\begin{split}
		 		 				  		  		 	f(xyz_0)&=\rho(x(yz_0))=\rho(x) \chi(yz_0)= \rho(x)\chi(y)\chi(z_0)\\
                                                    &=f(x)\chi(y)\chi(z_0)+f(y)\chi(x)\chi(z_0),
		 		 				  		  		 \end{split}\end{equation*}
		 		 				  		  		 since $ \chi(x)=0$.  The case  $y \in   P_{\chi}$  and $x \in S\setminus I_{\chi}$ also can be treated similarly.
		 		 				  		  		
		 		 				  		  		 Now,  for    $x\in I_{\chi}$  , $y\in I_{\chi}$ we have $xyz_0=x(yz_0) \in I_{\chi}^2 \subset  I_{\chi}\setminus P_{\chi}$. Then
		 		 				  		  		 \begin{equation*}
		 		 				  		  		 f(xyz_0)= 0=f(x)\chi(y)\chi(z_0)+f(y)\chi(x)\chi(z_0)
		 		 				  		  		 \end{equation*}
		 		 				  		  		 since $\chi(x)=\chi(y)=0$. Which proves   (\ref{bb}).
		 		 				  		  		
		 		 				  		  		  Next we prove that $f$ is abelian.  By computing $f((xyu)z_{0}z_{0})$ in two different ways, and using (\ref{bb}) we have for all  $x,y,u\in S$, by
		 		 				  		  \begin{equation*}\begin{split}
		 		 				  		  		 	f((xyu)z_{0}z_{0})&=f(xyu)\chi(z_0)\chi(z_0)+f(z_0)\chi(xyu)\chi(z_0)\\
		 		 				  		  		 	&=f(x)\chi(yuz_0)\chi(z_0)+f(yuz_0)\chi(x)\chi(z_0)\\
		 		 				  		  		 	&=\chi(x)\chi(z_0)[f(y)\chi(u)\chi(z_0)+f(u)\chi(y)\chi(z_0)]\\&\quad+f(x)\chi(yuz_0)\chi(z_0)\\
		 		 				  		  		 	&=\chi(x)\chi(z_0)[f(u)\chi(y)\chi(z_0)+f(y)\chi(u)\chi(z_0)]\\&\quad+f(x)\chi(uyz_0)\chi(z_0)\\
		 		 				  		  		 	&=f(x)\chi(uyz_0)\chi(z_0)+f(uyz_0)\chi(x)\chi(z_0)\\
		 		 				  		  		 	&=f(xuyz_0z_0)\\
		 		 				  		  		 	&=f(xuy)\chi( z_0)\chi(z_0)+f( z_0)\chi(xuy)\chi(z_0),
		 		 				  		  		 \end{split}\end{equation*}
		 		 				  		  		 then we deduce that  $f(xyu)\chi(z_0)\chi(z_0)=f(xuy)\chi( z_0)\chi(z_0)$ which implies $f(xyu)=f(xuy)$ for all $x,y,u \in S$, since $\chi(z_0)\neq0$. Next we have
		 		 				  		  		   \begin{equation*}\begin{split}
		 		 				  		  		   	f((xy)z_0^2 z_{0})&=f(xy)\chi(z_0^2)\chi(z_0)+f(z_0^2)\chi(xy)\chi(z_0)\\
		 		 				  		  		   	&=f(xyz_0)\chi(z_0)\chi(z_0)+f( z_0)\chi(xyz_0)\chi(z_0)\\
		 		 				  		  		   	&= f(yxz_0)\chi( z_0)\chi(z_0)+f(z_0) \chi(yxz_0)\chi(z_0)\\
		 		 				  		  		   	&=f((yx)z_0^2 z_{0})\\
		 		 				  		  		   	&=f(yx)\chi( z_0^2)\chi(z_0)+f(z_0^2)\chi(yx)\chi(z_0,)\\
		 		 				  		  		   \end{split}\end{equation*}
		 		 				  		  		 which implies that
		 		 				  		  		 \begin{equation*}
		 		 				  		  		 f(xy)\chi(z_0^2)\chi(z_0)=f(yx)\chi( z_0^2)\chi(z_0),
		 		 				  		  		    \end{equation*}
		 		 				  		  		   and therefore we obtain $ f(xy)= f(yx)$  for all $x,y \in S$. This proves
		 		 				  		  		     the centrality of $f$. Thus $f$ is abelian.
		 		 				  		  		
		 		 				  		  		     Finally, suppose that $S$ is a topological semigroup.      Putting $y=a$ in (\ref{bb}) such that $f(a)\neq0$   (because $f\neq 0$ by assumption) we get that
		 		 				  		  		  \begin{equation*}
		 		 				  		  		  \chi(x)= \frac{f(xaz_0)-f(x)\chi(a)\chi(z_0)}{f(a)\chi(z_0)}
		 		 				  		  		    \end{equation*}
		 		 				  		  		  since $\chi(z_0)\neq 0$.  This  gives the continuity of $\chi$. Therefore $A:=f/\chi- A(z_0)$ is continuous on $S\setminus I_{\chi}$ since $\chi\neq 0 $ on $S\setminus I_{\chi}$. Here also  $\rho $ is continuous by restriction. This completes the proof of Proposition \ref{t2}. \end{proof}Now, we are ready to describe the solutions of the functional equation (\ref{3}) on semigroups.
		 \begin{thm}
		 	\label{t1}
		     The solutions $f,g:S\longrightarrow \mathbb{C}$ of the functional equation (\ref{3}) are the following pairs of   functions\\
		     (1) $f=0$ and $g$ arbitrary. \\
		     (2) There exist a constant $b\in \mathbb{C}^*$ and an exponential function $\chi $   on $S$ with $\chi(z_0)\neq0$
		      such that
		      \begin{equation*}
		      f=\dfrac{\chi(z_0)}{2b}\chi  \hspace{0.3cm}  and \hspace{0.3cm} g= \dfrac{\chi(z_0)}{2}\chi.
		       \end{equation*}
		     (3) There exist a constant $\delta \in \mathbb{C}^*$ and two different  exponentials functions   $\chi_1$ and $\chi_2$  on $S$
		      with $\chi_2(z_0)=-\chi_1(z_0) \neq0$ such that
		      \begin{equation*}
		      f= \delta (\chi_1+\chi_2)   \hspace{0.3cm}  and \hspace{0.3cm} g=\dfrac{\chi_1(z_0)}{2}(\chi_1-\chi_2).
		      \end{equation*}
		     (4) There exist a constant $c \in \mathbb{C}^*$ and two different  exponentials functions $\chi_1$
		      and $\chi_2$ on $S$ with $\chi_2(z_0)=\chi_1(z_0)\neq0$ such that
		       \begin{equation*}
		      f= c(\chi_1-\chi_2)   \hspace{0.3cm}  and \hspace{0.3cm}  g=\dfrac{\chi_1(z_0)}{2} (\chi_1 + \chi_2).
		       \end{equation*}
		      (5) There exist a constant $\gamma \in \mathbb{C}^*$ and two different exponentials functions $\chi_1$  and $\chi_2$  on $S$ with $\chi_1(z_0) \neq0$,
		         $\chi_2(z_0)\neq0$ and  $\chi_1(z_0)\chi_1 \neq  \chi_2(z_0)\chi_2$  such that
		        \begin{equation*}
		       f= \dfrac{\chi_1(z_0)\chi_1-\chi_2(z_0)\chi_2}{2\gamma }   \hspace{0.3cm}  and \hspace{0.3cm}  g=\dfrac{\chi_1(z_0)\chi_1+\chi_2(z_0)\chi_2}{2}.
		       \end{equation*}
		         (6)   $S\neq S^2z_0 $    and we have
		         \begin{equation*}
		         g=0 \;  and  \;
		            f(x)=
		          \begin{cases}
		        f_{z_0}(x)\; for\; x\in S  \setminus S^2z_0 \\
		          0\; for\; x\in   S^2z_0,
		          \end{cases}
		           \end{equation*}
		          where   $f_{z_0}:S  \setminus S^2z_0  \longrightarrow \mathbb{C}$  is an arbitrary non-zero function.\\
		         (7) There exist    an  exponential  function $\chi$ on $S$, an additive
		         function $A: S \setminus$$I_{\chi}\longrightarrow \mathbb{C}$
		         and a function $\rho: P_{\chi}\longrightarrow \mathbb{C}$ with $\rho$ is the restriction of $f$ to $P_{\chi} $ and
		         $\chi(z_0)\neq0$  such that
		         \begin{equation*}
		    		          g=\chi(z_0)\chi \; and \;
		         f(x)=
		         \begin{cases}
		         \chi(x)(A(x)+A(z_0)) \hspace{0.5cm} for\hspace{0.3cm} x\in S \backslash I_{\chi} \\
		         \rho(x)\hspace{3cm} for\hspace{0.3cm} x\in    P_{\chi} \\
		         0 \hspace{3.5 cm}for\hspace{0.3cm} x\in I_{\chi}\setminus P_{\chi}.
		         \end{cases}
		         \end{equation*}
		
		          In addition, in the case (7)    $f$ satisfies the following conditions:\\
		          (I) $f(xy) = f(yx)=0 $ for all $ x \in I_{\chi} \setminus P_{\chi} $ and $ y \in S \setminus I_{\chi}$.\\
		          (II) If $x \in \{up, pv, upv\}
		          		$ with $p \in P_{\chi} $ and $u,v \in S \setminus I_{\chi}$, then  $x \in P_{\chi} $ and we have respectively
		          			$\rho(x)=	\rho(p) \chi(u)$, $\rho(x)=	\rho(p) \chi(v)$ or $\rho(x)=	\rho(p) \chi(uv)$.
		          			
		          			 Moreover if $S$ is a topological semigroup and $f \in C(S)$, then $g, \chi,\chi_1, \chi_2 \in C(S), A \in C(S \setminus  I_{\chi})$,
		          				 $f_{z_0}\in C(S  \setminus S^2z_0)$ and
		          					$\rho \in C(P_{\chi})$.\\ Note that $f$ and $g$ are abelian in  cases (2)-(5) and (7).
		          	\end{thm}	  				  		
		 		 \begin{proof}				  	
		 		It is easy to check that if $f=0$, then $g$ can be chosen arbitrary, so we have  the solution family (1). From now on  we may assume that $f\neq0$.\\ If $f$ and $g$ are linearly dependent, then there exists a constant $ b \in \mathbb{C}$ such that $g=bf$ and therefore   equation (\ref{3}) becomes
		 		$
		 	f(xyz_0)=2bf(x)f(y),\; x,y \in S.
		 		  $\\
		 		 If $b=0$ then $g=0$ and $f(xyz_0)=0$ for all $x, y \in S$, so necessarily  we have $S\neq S^2z_0$, because $ f\neq0$, and then we have the  solution family (6)  for an arbitrary non-zero function $f_{z_{0}}:S\setminus S^2z_0\longrightarrow   \mathbb{C}$.
		 		 If $b\neq0$, then by \cite[Proposition 16]{i} we infer that there exists a non-zero multiplicative function $\chi$ on $S$ such that $2bf=: \chi(z_0)\chi$, which gives
		 		 \begin{center}
		 		 $f=\dfrac{\chi(z_0)}{2b}\chi$ \; and \;   $g=\dfrac{\chi(z_0)}{2}\chi$,
		 		\end{center}
		 		  and this proves that $(f, g)$ belongs
		 		 to the solution family (2).\\ For the rest of the proof we assume that $f$ and $g$ are linearly  independent. Then, according to Lemma \ref{l2} (ii)-(a),  we get that $g(z_0)\neq 0 $.
		 		
		 		 To complete the proof we split it into   two   cases according
		 		  to whether $f(z_0)=0 $   or $f(z_0) \neq0$.
		 		
		 		  {Case 1}. Suppose $f(z_0)=0 $. From Lemma \ref{l1} we get that $f$ and $g$ satisfy the following functional equation
		 		   \begin{equation}
		 		   \label{9}
		 		   g(z_0^2)f(xy)= g(z_0)f(x)g(y)+g(z_0)f(y)g(x)-f(z_0^2)g(xy),\;x,y\in S,
		 		   \end{equation}
		 		
		 		     Subcase 1.1: Suppose $f(z_0^2)=0$, then by lemma \ref{l2}  (ii)-(b) we get that     $g(z_0^2)\neq0$. Therefore
		 		        equation (\ref{9}) can be rewritten as follows
		 		        \begin{center}
		 		     $
		 		      f(xy)=  \alpha f(x)g(y)+ \alpha f(y)g(x), \; x,y\in S,
		 		     $
		 		    \end{center}
		 		      with $\alpha:=\dfrac{g(z_0)}{g(z_0^2)}\neq0$. Then the pair $(f, \alpha g)$ satisfies the sine addition law (\ref{1}). According to \cite[Theorem 4.1 (b)]{g}
		 		  we infer that there exist  two multiplicative functions $\chi_1$ and $\chi_2$ on $S$ such that
		 		      $
		 		       \alpha g=\dfrac{\chi_1+\chi_2}{2}.
		 		      $\\
		 		      (i) For the case $\chi_1\neq \chi_2$, there exists a constant $c\in \mathbb{C}^{\ast}$  such that
		 		       \begin{center}
		 		       $
		 		     f=c(\chi_1 - \chi_2)$ \; and\;  $g=\dfrac{\chi_1+\chi_2}{2\alpha}.
		 		     $
		 		      \end{center}
		 		       Furthermore, if $\chi_1=0$ or $\chi_2 =0$, we get  $f,g$ linearly dependent, which contradict the assumption that $f$ and $g$ are linearly independent, so  $\chi_1$ and $\chi_2$ are exponentials.\\
		 		     By using (\ref{3}) we get
		 		    \begin{equation*}\begin{split}
		 		     0&=f(xyz_0)-f(x)g(y)-f(y)g(x)\\
                     &=c(\chi_1 - \chi_2)(xyz_0) -c(\chi_1 - \chi_2)(x)\left(\dfrac{\chi_1+\chi_2}{2\alpha}\right)(y)\\&\quad-c(\chi_1 - \chi_2)(y)\left(\dfrac{\chi_1+\chi_2}{2\alpha}\right)(x)\\
		 		   &=\left(c\chi_1(z_0)-\dfrac{c}{\alpha}\right)\chi_1(xy)-\left(c\chi_2(z_0)-\dfrac{c}{\alpha}\right)\chi_2(xy),
		 		     \end{split}\end{equation*}
		 		     for all $x,y\in S$. Since $\chi_1$ and $\chi_2$ are different we get by Lemma \ref{ui} (a) that
		 		     \begin{center}
		 		     $
		 		    c\chi_1(z_0)-\dfrac{c}{\alpha}=0 $ \;and \;$ c\chi_2(z_0)-\dfrac{c}{\alpha}=0  ,
		 		   $ \end{center}
		 		       which gives $\chi_1(z_0)=\chi_2(z_0)=\dfrac{1}{\alpha} $, since $ c\neq0$. Then   $g$ becomes
		 		     $g=  \chi_1(z_0)(\chi_1+\chi_2)  / 2$. This gives that the solution is
		 		     in part (4).
		 		         (ii) For the case $\chi_1 =\chi_2=:\chi$  with  $\chi$ exponential, we have  $\alpha g=\chi$ and $f$ has the form (3) of \cite[Theorem 3.1]{d}.  Since $0 \neq \alpha g(z_0) =\chi(z_0)$ in the present Subcase 1.1,  we   get that $z_0 \in S\setminus I_{\chi}$. Now for all $x,y$  $\in S\setminus I_{\chi}$ we have $ xyz_0\in S\setminus I_{\chi}$   and by using (\ref{3}) we get that
		 		
		 		         	\begin{equation*}\begin{split}0&=f(xyz_0)-f(x)g(y)-f(y)g(x)\\
		 		         	 &= A(xyz_0)\chi(xyz_0)-\dfrac{A(x)\chi(xy)}{\alpha }-\dfrac{A(y)\chi(xy)}{\alpha }\\
		 		         	&= A(xy)\chi(xy)\left( \chi(z_0)-\dfrac{1}{\alpha }\right) +A(z_0)\chi(z_0)\chi(xy).\end{split}\end{equation*}
		 		         	 If $A\neq0$, then by Lemma \ref{ui}.(b) we get   that
		 		         	 \begin{center}
		 		         	  $\chi(z_0)-\dfrac{1}{\alpha}=0$\; and \; $A(z_0)\chi(z_0)=0.
		 		         	  $\end{center}
		 		         	  Then we deduce that $\chi(z_0)=\dfrac{1}{\alpha}\neq0$ and therefore $A(z_0)=0$. So we are in solution family (7).\\If $A=0,$ then  $f= A\chi=0$ on $S\setminus I_{\chi}$ and necessarily we have $\rho \neq0 $ because $f\neq0$. Next, let  $x\in P_{\chi}$  and $y\in S\setminus I_{\chi}$, then $xyz_0 \in P_{\chi}$  by    Lemma \ref{eb}.    By using (\ref{3}) and   Lemma \ref{eb} we get that
		 		         	  \begin{center}
		 		           $\rho(x)\chi(y)\chi(z_0)=\rho(xyz_0) =f(xyz_0) =f(x) g(y)+f(y)g(x)
		 		         	= \dfrac{1  }{\alpha} \rho(x) \chi(y),$
		 		         	\end{center}
		 		         	 since $f=0$ on $S\setminus I_{\chi}$ and $\chi(x)=0$.   This implies that $\chi(z_0)=1/\alpha $ because $ \rho\neq0$ and $\chi(y)\neq0$.	
		 		         	  So we have  the special case of the solution family (7).
		 		         	
		 		         	  Subcase 1.2:	  Suppose $f(z_0^2)\neq 0$. From lemma \ref{l2} (i) we read that $g(z_0^2)= 0$.    Therefore   Eq.(\ref{9}) yields
		 		     \begin{equation*}
		 		      \label{bhg}
		 		    g(xy)=\beta g(x)f(y)+\beta g(y)f(x), \; x,y \in S,
		 		     \end{equation*}
		 		    with $\beta:= \dfrac{g(z_0)}{f(z_0^2)}\neq0  $ since $g(z_0)\neq 0$.
		 		   This means that
		 		     the pair $(g,  \beta f )$ satisfy the sine addition formula (\ref{1}) with $g\neq0$ since $f$ and $g$ are linearly independent. Applying   \cite[Theorem 4.1 (b)]{g}   we infer that there    exist   two   multiplicative functions $\chi_1$ and $\chi_2$ on $S$ such that:
		 		   $
		 		    \beta f =\dfrac{\chi_1+\chi_2}{2}.
		 		  $ \\(i) In  the case $\chi_1\neq \chi_2$, there    exists a constant $c\in \mathbb{C}^{*}$ such that
		 		  \begin{center}
		 		    $g=c(\chi_1 - \chi_2)$ \; and  \; $
		 		     f=\dfrac{\chi_1+\chi_2}{2\beta }. $
		 		     \end{center}
		 		    Furthermore, $\chi_1 $ and $\chi_2 $ are exponentials, because $f$ and $g$ are linearly independent.
		 		     A small computation based on (\ref{3})  shows that
		 		     \begin{equation*}
		 		     \left( \dfrac{c}{\beta}-\dfrac{\chi_1(z_0)}{2\beta}\right) \chi_1(xy)- \left( \dfrac{c}{\beta}+\dfrac{\chi_2(z_0)}{2\beta}\right) \chi_2(xy)=0.
		 		     \end{equation*}
		 		     Since $\chi_1\neq \chi_2$   we get by Lemma \ref{ui} (a) that $\chi_1(z_0)=-\chi_2(z_0)=2c$. So we are in the part (3) of our statement, with $\delta:=1/2 \beta$.		 		
		 		    (ii)  In the case $\chi_1 =\chi_2:=\chi $ with $\chi$ exponential, we have $\beta f=\chi $  and $g$  has the form (3) of \cite[Theorem 3.1]{d}. Since $0=\beta f(z_0)=\chi(z_0)$ by the hypothesis of this case 1, we have necessarily $z_0 \in I_{\chi}$. On the other hand since  $g(z_0)\neq0$ we get  that $ z_0 \in   P_{\chi}$, because  $g =0$ on $I_{\chi}\setminus P_{\chi}$ and $I_{\chi}= P_{\chi}\cup (I_{\chi}\setminus P_{\chi} )$. Now let $ x, y\in S\setminus I_{\chi}$, then by Lemma \ref{eb}  we get $xyz_0 \in  P_{\chi} $. Therefore by using (\ref{3}) we obtain
		 		      \begin{equation*}\begin{split}0&=f(xyz_0)-f(x)g(y)-f(y)g(x)\\
		 		      	&=\rho(xyz_0)-\dfrac{1}{\beta}\chi(x)A(y) \chi(y)
		 		      	-\dfrac{1}{\beta}\chi(y)A(x)\chi(x)\\
		 		      		&=\rho( z_0) \chi(xy)-\dfrac{1}{\beta} A(y)\chi(xy)
		 		      	 -\dfrac{1}{\beta} A(x)\chi(xy) A(x)
		 		      	 \\
		 		      	&=\rho( z_0) \chi(xy)-\dfrac{1}{\beta}\chi(xy)A(xy),
		 		      \end{split}\end{equation*}
		 		    which gives $\rho( z_0)=   \dfrac{1}{\beta} A(xy)$. Therefore $A=0$  and $\rho( z_0)=0$ which is a contradiction, because $\rho( z_0)=g(z_0)\neq0$ by Lemma \ref{l2} (ii)-(a). So this case is excluded.

		 		   {Case 2}. Suppose $f(z_0)\neq0 $. Multiplying (\ref{3}) by $\pm \gamma$  where  $\gamma \in \mathbb{C}$ is the constant in
		 		   the identity (\ref{63}) and  adding the result to the identity (\ref{63}) we get
		 			 \begin{equation*}(g+\gamma f)(xyz_0)=(g+\gamma f)(x)	(g+\gamma f)(y)
		 		         \end{equation*} and
		 		      \begin{equation*}
		 		      (g-\gamma f)(xyz_0)=(g-\gamma f)(x)	(g-\gamma f)(y).
		 		   \end{equation*}
		 		      According to \cite[Proposition 16]{i}, we infer that there exist  two exponentials functions $\chi_1$ and $\chi_2$  on $S$ such that
		 		     \begin{equation*}
		 		       \chi_1(z_0)\chi_1:= g+\gamma f	 \hspace{0.3cm} and \hspace{0.3cm} \chi_2(z_0)\chi_2:=	g-\gamma f,
		 		    \end{equation*}
		 		     with $ \chi_1(z_0) \neq0 $ and $ \chi_2(z_0)  \neq0 $, because $g-\gamma f \neq0$ and $g+\gamma f \neq0$ since $f$ and $g$ are linearly independent. Then we get
		 		     \begin{center}
		 		     $g= \dfrac{\chi_1(z_0)\chi_1+\chi_2(z_0)\chi_2}{2}.$
		 		    \end{center}
		 		      If $\chi_1(z_0)\chi_1\neq\chi_2(z_0)\chi_2$, then    $2\gamma f=\chi_1(z_0)\chi_1-\chi_2(z_0)\chi_2$ and $\gamma\neq0$. Therefore
		 		      \begin{center}
		 		      $f= \dfrac{\chi_1(z_0)\chi_1-\chi_2(z_0)\chi_2 }{2\gamma} $.
		 		     \end{center}
		 		    This solution falls in  solution family (5).\\If $\chi_1(z_0)\chi_1=\chi_2(z_0)\chi_2 $,  then   we get by Lemma \ref{ui} (a)   that $ \chi_1=\chi_2$,  since  $ \chi_1(z_0)\neq0$ and $\chi_2(z_0)\neq0$. This allows us to define $\chi:= \chi_1= \chi_2$ with $\chi$    exponential such that $\chi(z_0)\neq0$, and we obtain $  g= \chi(z_0)\chi$. Therefore Eq.(\ref{3}) becomes
		 		      \begin{equation}
		 		       \label{99}
		 		       f(xyz_0)=f(x)\chi(z_0)\chi(y)+f(y)\chi(z_0)\chi(x).
		 		      \end{equation}
		 		       This equation was solved in Proposition \ref{t2}.    The solutions in this case occur in  part (7).
		 		
		 		       That $f$ and $g$ are abelian in the cases (2)-(5)    can be derived from  the solutions formulas above and the case (7) is  established in Proposition \ref{t2}.
		 		
		 		       	Conversely it is easy to check that the formulas of $f$ and $g$ listed in  Theorem 4.4 define  solutions of (\ref{3}).
		 		       	
		 		       	Finally, suppose that $S$ is a topological semigroup. The continuity of cases (1)-(5) follows from \cite[Theorem 3.18]{g} since $f\neq0$,  $c\neq0,b\neq0,  \delta\neq0, \chi(z_0)\neq0,\chi_1(z_0 )\neq0$, $\chi_2(z_0)\neq0$ and also  $\chi_1$ and $\chi_2$ are different. For the case (6), $f_{z_{0}}$ is continuous by restriction.
		 		        The case (7) follows  directly from Proposition \ref{t2}.
		 		         \end{proof}
		 		         In the following Proposition we show that the functional equation (\ref{5}) has only trivial solutions on semigroup.
		 		          \begin{prop}
		 		          	\label{df}
		 		          	The solutions $f,g:S\longrightarrow \mathbb{C}$ of the functional equation (\ref{5}) are the following pairs of   functions listed below\\
		 		          	(1) $f=0$ and $g$ arbitrary \\
		 		          	(2)  $S\neq S^2z_0 $ and  there exists a constant $c\in \mathbb{C}$  such that
		 		          	
		 		          	 \begin{equation*}
		 		          	 g=cf \hspace{0.4cm}  and  \hspace{0.4cm}
		 		          	 f(x)=
		 		          	 \begin{cases}
		 		          	 f_{z_0}(x)\hspace{0.6cm} for\hspace{0.3cm} x\in S  \setminus S^2z_0 \\
		 		          	 0\hspace{1.5cm} for\hspace{0.3cm} x\in   S^2z_0,
		 		          	 \end{cases}
		 		          	 \end{equation*}
		 		          	where   $f_{z_0}:S  \setminus S^2z_0  \longrightarrow \mathbb{C}$  is an arbitrary non-zero function.
		 		          \end{prop}
		 		         \begin{proof}
		 		         	If $f=0$ then $g$   can be chosen arbitrary. So from now we may assume that $f\neq0$. If $f$ and $ g $ are linearly  dependent, then there exists a constant $c \in \mathbb{C}$ such that $g=c f$. Therefore Eq.(\ref{5})  becomes $f(xyz_0)=0$ for all $ x,y \in S$, so   necessarily we have $S\neq S^2z_0$, because $f\neq 0$. Then we get the    solution family (2). \\If  $f$ and $ g $ are linearly independent. By interchanging $x$ and $y$ in (\ref{5}) we get
		 		         	\begin{center}
		 		       $f(xyz_0)=-f(yxz_0)$, \;    $x,y \in S$.
		 		   \end{center}
		 		   Using this we obtain
		 		       \begin{center}
		 		        $f(xyzz_0)=-f(zxyz_0)=f(yzxz_0)=-f(xyzz_0)$,\;  $x,y,z \in S,$
		 		    \end{center}
		 		         which gives $f(xyzz_0)=0$  for all  $x,y,z \in S$. Therefore by using (\ref{5})	 we get that
		 		          \begin{center}
		 		          $f(xyzz_0)=f(xy)g(z)-f(z)g(xy)=0, $ for all $x,y,z \in S,$
		 		       \end{center}
		 		        which implies
 \begin{equation}
 \label{lh}
 f(xy)g(z)=f(z)g(xy),  \; x,y,z \in S.
 \end{equation}
If there exist $x_0, y_0 \in S$ such that $f(x_0y_0)\neq0$, then (\ref{lh}) with $(x,y)=(x_0,y_0)$ contradicts that $f $ and $g$ are linearly independent. Thus
$f = 0$ on $S^2$. In particular $f(xyz_0) = f((xy)z_0) = 0$, so by (\ref{5})  $f(x)g(y) =
f(y)g(x)$ for all $x,y\in S$ which contradicts that $f$ and $g$ are linearly independent. This completes the proof of Proposition \ref{df}.
		 		         	\end{proof}
		 		         	

		 		        \end{document}